\documentclass[a4paper,reqno,11pt]{amsart}

\usepackage{amsmath}
\usepackage{amssymb}
\usepackage{amsthm}
\usepackage{latexsym}
\usepackage[all]{xypic}
\usepackage{amsfonts}
\usepackage{mathrsfs}
\usepackage{graphicx,color}
\usepackage{xcolor}
\usepackage{tikz}
\usepackage{graphics}
\usetikzlibrary[shapes]
\xyoption{all}

\usepackage{anysize,hyperref}
\marginsize{30mm}{30mm}{23mm}{23mm} 

\usepackage{soul}
\setstcolor{red}

\newtheorem{theorem}{Theorem}[section]
\newtheorem{lemma}[theorem]{Lemma}
\newtheorem{corollary}[theorem]{Corollary}
\newtheorem{proposition}[theorem]{Proposition}
\theoremstyle{definition}
\newtheorem{definition}[theorem]{Definition}
\newtheorem{definition-proposition}[theorem]{Definition-Proposition}
\newtheorem{remark}[theorem]{Remark}
\newtheorem{example}[theorem]{Example}

\def\C{\mathcal{C}}
\def\D{\mathcal{D}}

\def\T{\mathcal{T}}

\def\W{\mathscr{W}}

\def\X{\mathscr{X}}
\def\Y{\mathscr{Y}}
\def\Z{\mathcal {Z}}
\def\I{\mathcal {I}}

\def\M{\mathcal{M}}

\def \text{\mbox}

\providecommand{\add}{\mathop{\rm add}\nolimits}%
\providecommand{\Hom}{\mathop{\rm Hom}\nolimits}%

\def\cO{\mathscr{O}}
\def\cT{\mathscr{T}}

\def\clap#1{\hbox to 0pt{\hss#1\hss}}

\newcommand{\leftsup}[2]{{\vphantom{#2}}^{#1}{#2}}
\providecommand{\nc}{\mathop{\rm nc}\nolimits}%

\usetikzlibrary{arrows,decorations.pathmorphing,decorations.pathreplacing,positioning,shapes.geometric,shapes.misc,decorations.markings,decorations.fractals,calc,patterns}
\numberwithin{equation}{section}

\begin{document}

\title{Cotorsion pairs in $(d+2)$-angulated categories}

\author[Chang]{Huimin Chang}
\address{
Department of Applied Mathematics,
The Open University of China,
100039 Beijing,
P. R. China
}
\email{changhm@ouchn.edu.cn}
\author[Zhou]{Panyue Zhou$^\ast$}
\address{School of Mathematics and Statistics, Changsha University of Science and Technology, 410114 Changsha, Hunan,  P. R. China}
\email{panyuezhou@163.com}

\thanks{$^\ast$Corresponding author.}

\begin{abstract}
Let $\C$ be a $(d+2)$-angulated category. In this paper, we define the notions of cotorsion pairs and weak cotorsion pairs in $\C$, which are generalizations of the classical cotorsion pairs in triangulated categories. As an application, we give a geometric characterization of weak cotorsion pairs in $(d+2)$-angulated cluster categories of type $A$. Moreover, we prove that any mutation of a (weak) cotorsion pair in $\C$ is again a (weak) cotorsion pair. When $d=1$, this result generalizes the work of Zhou and Zhu on classical cotorsion pairs in triangulated categories.
\end{abstract}

\subjclass[2020]{18E40; 05E10; 18G80}

\keywords{cotorsion pair; mutation; $(d+2)$-angulated category; cluster tilting subcategory}

\thanks{Huimin Chang is supported by the National Natural Science Foundation of China (Grant No. 12301047).
Panyue Zhou is supported by the National Natural Science Foundation of China (Grant No. 12371034) and by the Hunan Provincial Natural Science Foundation of China (Grant No. 2023JJ30008). }

\maketitle

\tableofcontents

\section{Introduction}
The concept of torsion pairs in abelian categories was first introduced by Dickson \cite{D}, with its triangulated version later explored by Iyama and Yoshino \cite{IY}. Nakaoka \cite{N} then introduced cotorsion pairs in triangulated categories, aiming to unify the abelian structures that arise from $t$-structures \cite{BBD} and from cluster tilting subcategories \cite{KR,KZ,IY}. In triangulated categories, torsion pairs and cotorsion pairs are closely related, as they can be transformed into each other by shifting the torsion-free components. Therefore, considering torsion pairs and cotorsion pairs in triangulated categories is essentially equivalent.

In \cite{GKO}, Geiss, Keller, and Oppermann introduced a higher-dimensional generalization of triangulated categories, known as $(d+2)$-angulated categories. These categories extend the classical notion of triangulated categories, with the latter being a special case when $d=1$. $(d+2)$-angulated categories often arise in the context of $d$-cluster tilting subcategories in triangulated categories.

A significant motivation for studying mutation stems from cluster algebras. Fomin and Zelevinsky \cite{FZ} introduced cluster algebras to provide an algebraic and combinatorial framework for understanding the positive and canonical bases of quantum groups. In their work, the mutation of clusters was defined in the context of cluster algebras. To categorize cluster algebras, Buan et al. \cite{BMRRT} introduced cluster categories, which Keller \cite{K} later showed to be triangulated. In these categories, rigid indecomposable objects correspond to cluster variables, while cluster tilting subcategories correspond to clusters. Similar to the mutation of clusters, the mutation of cluster tilting subcategories involves replacing one indecomposable object with a unique alternative such that the result is again a cluster tilting subcategory. For more details, see \cite{BMRRT}.
As a further generalization, the mutation of cluster tilting subcategories in arbitrary Krull-Schmidt $K$-linear triangulated categories was studied in \cite{BIRS,IY,P}. Iyama and Yoshino \cite{IY} introduced a broader concept of mutation in triangulated categories, which encompasses the mutation of cluster tilting subcategories as a special case. Zhou and Zhu \cite{ZZ} later proved that any mutation of a torsion pair remains a torsion pair.
After the introduction of $(d+2)$-angulated categories, Lin \cite{L} defined mutation pairs in these categories and proved that, given such a mutation pair, the corresponding quotient category naturally inherits a $(d+2)$-angulated structure.

In this paper, we define cotorsion pairs and weak cotorsion pairs in $(d+2)$-angulated categories, generalizing the classical notion of cotorsion pairs in triangulated categories. It is worth noting that the concept of cotorsion pairs in $(d+2)$-angulated categories introduced here is stricter than the torsion class defined by J{\o}rgensen in \cite{J}. As an application, we provide a geometric characterization of weak cotorsion pairs in $(d+2)$-angulated cluster categories of type $A$. Furthermore, we demonstrate that any mutation of a (weak) cotorsion pair in $(d+2)$-angulated cluster categories remains a (weak) cotorsion pair. When $d=1$, this result extends the work of Zhou and Zhu on classical cotorsion pairs in triangulated categories.

This paper is organized as follows. In Section 2, we provide an overview of $(d+2)$-angulated categories, mutation pairs within these categories, and recall the construction of quotient categories from \cite{L}. In Section 3, we define cotorsion pairs in $(d+2)$-angulated categories and present a geometric characterization of weak cotorsion pairs in $(d+2)$-angulated cluster categories of type $A$. In Section 4, we define mutation of cotorsion pairs, examine their compatibility with cotorsion pairs, and prove that any mutation of a (weak) cotorsion pair in $(d+2)$-angulated cluster categories remains a (weak) cotorsion pair.
\vspace{2mm}

\hspace{-4mm}{\bf Conventions} 
Let $\C$ be a $(d+2)$-angulated category with a $d$-suspension functor $\Sigma^d$.
For objects $X$ and $Y$ in $\C$, we denote the set of morphisms from
$X$ to $Y$ by ${\rm Hom}_{\C}(X,Y)$ or $\C(X,Y)$. The composition of two morphisms
$f\in {\rm Hom}_{\C}(X,Y)$ and $g\in{\rm Hom}_{\C}(Y,Z)$
is written as $g\circ f\in{\rm Hom}_{\C}(X,Z)$.
When we say that $\X$ is a subcategory of $\C$, we always mean that $\C$
is a full subcategory which is closed under isomorphisms, direct sums and
direct summands.
We denote by $\X^\perp$ (resp. $^\perp\X$) the subcategory consisting of objects  $M\in\C$
such that $\Hom_{\C}(\X,M)=0$ (resp. $\Hom_{\C}(M,\X)=0$). For an object $X\in\C$, $\add X$ means the additive closure of $X$.
\vspace{2mm}

\section{Preliminaries}
We recall the definitions of $(d+2)$-angulated categories, mutation pairs in $(d+2)$-angulated categories, and the construction of quotient categories from \cite{L}.

\subsection{$(d+2)$-angulated categories}
In this subsection, we recall some definitions and basic properties of $(d+2)$-angulated categories from \cite{GKO}.
Let $\C$ be an additive category with an automorphism $\Sigma^d:\C\rightarrow\C$,  where $d$ is an integer no less than one.
\vspace{1mm}

A $(d+2)$-$\Sigma^d$-$sequence$ in $\C$ is a sequence of objects and morphisms
$$A_0\xrightarrow{f_0}A_1\xrightarrow{f_1}A_2\xrightarrow{f_2}\cdots\xrightarrow{f_{d-1}}A_n\xrightarrow{f_d}A_{d+1}\xrightarrow{f_{d+1}}\Sigma^d A_0.$$
Its {\em left rotation} is the $(d+2)$-$\Sigma^d$-sequence
$$A_1\xrightarrow{f_1}A_2\xrightarrow{f_2}A_3\xrightarrow{f_3}\cdots\xrightarrow{f_{d}}A_{d+1}\xrightarrow{f_{d+1}}\Sigma^d A_0\xrightarrow{(-1)^{d}\Sigma^d f_0}\Sigma^d A_1.$$
A \emph{morphism} of $(d+2)$-$\Sigma^d$-sequences is  a sequence of morphisms $\varphi=(\varphi_0,\varphi_1,\cdots,\varphi_{d+1})$ such that the following diagram
$$\xymatrix{
A_0 \ar[r]^{f_0}\ar[d]^{\varphi_0} & A_1 \ar[r]^{f_1}\ar[d]^{\varphi_1} & A_2 \ar[r]^{f_2}\ar[d]^{\varphi_2} & \cdots \ar[r]^{f_{d}}& A_{d+1} \ar[r]^{f_{d+1}}\ar[d]^{\varphi_{d+1}} & \Sigma^d A_0 \ar[d]^{\Sigma^d \varphi_0}\\
B_0 \ar[r]^{g_0} & B_1 \ar[r]^{g_1} & B_2 \ar[r]^{g_2} & \cdots \ar[r]^{g_{d}}& B_{d+1} \ar[r]^{g_{d+1}}& \Sigma^d B_0
}$$
commutes, where each row is a $(d+2)$-$\Sigma^d$-sequence.
It is an {\em isomorphism} if $\varphi_0, \varphi_1, \cdots, \varphi_{d+1}$ are all isomorphisms in $\C$.

\begin{definition}\cite[Definition 2.1]{GKO}
A $(d+2)$-\emph{angulated category} is a triple $(\C, \Sigma^d, \Theta)$, where $\C$ is an additive category, $\Sigma^d$ is an automorphism of $\C$ ($\Sigma^d$ is called the $d$-suspension functor), and $\Theta$ is a class of $(d+2)$-$\Sigma^d$-sequences (whose elements are called $(d+2)$-angles), which satisfies the following axioms:
\begin{itemize}
\item[\textbf{(N1)}]
\begin{itemize}
\item[(a)] The class $\Theta$ is closed under isomorphisms, direct sums and direct summands.

\item[(b)] For each object $A\in\C$, the trivial sequence
$$ A\xrightarrow{1_A}A\rightarrow 0\rightarrow0\rightarrow\cdots\rightarrow 0\rightarrow \Sigma^dA$$
belongs to $\Theta$.

\item[(c)] Each morphism $f_0\colon A_0\rightarrow A_1$ in $\C$ can be extended to $(d+2)$-$\Sigma^d$-sequence: $$A_0\xrightarrow{f_0}A_1\xrightarrow{f_1}A_2\xrightarrow{f_2}\cdots\xrightarrow{f_{d-1}}A_d\xrightarrow{f_d}A_{d+1}\xrightarrow{f_{d+1}}\Sigma^d A_0.$$
\end{itemize}
\item[\textbf{(N2)}] A $(d+2)$-$\Sigma^d$-sequence belongs to $\Theta$ if and only if its left rotation belongs to $\Theta$.

\item[\textbf{(N3)}] Each solid commutative diagram
$$\xymatrix{
A_0 \ar[r]^{f_0}\ar[d]^{\varphi_0} & A_1 \ar[r]^{f_1}\ar[d]^{\varphi_1} & A_2 \ar[r]^{f_2}\ar@{-->}[d]^{\varphi_2} & \cdots \ar[r]^{f_{d}}& A_{d+1} \ar[r]^{f_{d+1}}\ar@{-->}[d]^{\varphi_{d+1}} & \Sigma^d A_0 \ar[d]^{\Sigma^d \varphi_0}\\
B_0 \ar[r]^{g_0} & B_1 \ar[r]^{g_1} & B_2 \ar[r]^{g_2} & \cdots \ar[r]^{g_{d}}& B_{d+1} \ar[r]^{g_{d+1}}& \Sigma^d B_0
}$$ with rows in $\Theta$, the dotted morphisms exist and give a morphism of  $(d+2)$-$\Sigma^d$-sequences.

\item[\textbf{(N4)}] In the situation of {\bf (N3)}, the morphisms $\varphi_2,\varphi_3,\cdots,\varphi_{d+1}$ can be chosen such that the mapping cone
$$A_1\oplus B_0\xrightarrow{\left(\begin{smallmatrix}
                                        -f_1&0\\
                                        \varphi_1&g_0
                                       \end{smallmatrix}
                                     \right)}
A_2\oplus B_1\xrightarrow{\left(\begin{smallmatrix}
                                        -f_2&0\\
                                        \varphi_2&g_1
                                       \end{smallmatrix}
                                     \right)}\cdots\xrightarrow{\left(\begin{smallmatrix}
                                        -f_{d+1}&0\\
                                        \varphi_{d+1}&g_d
                                       \end{smallmatrix}
                                     \right)} \Sigma^n A_0\oplus B_{d+1}\xrightarrow{\left(\begin{smallmatrix}
                                        -\Sigma^d f_0&0\\
                                        \Sigma^d\varphi_1&g_{d+1}
                                       \end{smallmatrix}
                                     \right)}\Sigma^dA_1\oplus\Sigma^d B_0$$
belongs to $\Theta$.
   \end{itemize}
\end{definition}

We call $(\mathcal{C},\Sigma^d,\Theta)$ a $pre$-$(d+2)$-$angulated\ category$ if it satisfies $\textbf{(N1)}$, $\textbf{(N2)}$ and $\textbf{(N3)}$. Now we give an example of $(d+2)$-angulated category.

\begin{example}\label{def}
We recall the standard construction of $(d+2)$-angulated categories given by Geiss-Keller-Oppermann \cite[Theorem 1]{GKO}.
Let $\C$ be a triangulated category and $\mathcal{T}$ a $d$-cluster tilting subcategory which is closed under $\Sigma^{d}$, where $\Sigma$ is the shift functor of $\C$. Then $(\mathcal{T},\Sigma^{d},\Theta)$ is a $(d+2)$-angulated category, where $\Theta$ is the class of all sequences
$$A_0\xrightarrow{f_0}A_1\xrightarrow{f_1}A_2\xrightarrow{f_2}\cdots\xrightarrow{f_{d-1}}A_d\xrightarrow{f_d}A_{d+1}\xrightarrow{f_{d+1}}\Sigma^{d} A_0$$
such that there exists a diagram
$$\xymatrixcolsep{0.3pc}
 \xymatrix{& A_1 \ar[dr]\ar[rr]^{f_1}  &  & A_2  \ar[dr]  & & \cdots  & & A_{d} \ar[dr]^{f_{d}}      \\
A_0 \ar[ur]^{f_0} & \mid & \ar[ll]  A_{1.5}\ar[ur] & \mid &  \ar[ll]  A_{2.5} & \cdots & A_{d-1.5}\ar[ur] & \mid & \ar[ll] A_{d+1}   }$$
with $A_i\in\mathcal{T}$ for all $i\in\mathbb{Z}$, such that all oriented triangles are triangles in $\C$, all non-oriented triangles commute, and $f_{d+1}$ is the composition along the lower edge of the diagram.
\end{example}

The following lemma will be useful in what follows.

\begin{lemma}\cite[Theorem 3.1]{LZ}\label{thm1}
Let $(\mathcal{C},\Sigma^d,\Theta)$ be a pre-$(d+2)$-angulated category. Then $\Theta$ satisfies ${\rm \bf(N4)}$ if and only if $\Theta$ satisfies ${\rm\bf (N4\mbox{-}1)}$:

Let  $$\xymatrix{
X_1 \ar[r]^{f_1}\ar@{=}[d] & X_2 \ar[r]^{f_2}\ar[d]^{\varphi_2} & X_3 \ar[r]^{f_3} & \cdots \ar[r]^{f_{d+1}}& X_{d+2} \ar[r]^{f_{d+2}} & \Sigma^d X_1 \ar@{=}[d]\\
X_1 \ar[r]^{g_1} & Y_2 \ar[r]^{g_2} & Y_3 \ar[r]^{g_3} & \cdots \ar[r]^{g_{d+1}} & Y_{d+2} \ar[r]^{g_{d+2}}& \Sigma^d X_1\\
}$$ be a commutative diagram whose rows are $(d+2)$-angles. Then there exist morphisms $\varphi_i:X_i\rightarrow Y_i$ for $3\leq i\leq d+2$ such that $$
X_2\rightarrow X_3\oplus Y_2\rightarrow\cdots\rightarrow X_{d+2}\oplus Y_{d+1}\rightarrow Y_{d+2}\rightarrow \Sigma^d X_2
$$
is a $(d+2)$-angle. Note that we omit the morphisms between these objects.

\end{lemma}
\subsection{Mutation pairs in $(d+2)$-angulated categories}
In this subsection, we recall the concept of mutation pairs in
$(d+2)$-angulated categories from \cite{L}.

Let $\mathcal{C}$ be an additive category and $\mathcal{D}$ a subcategory of $\mathcal{C}$. A morphism $f:X\rightarrow Y$ in $\mathcal{C}$ is called $\mathcal{D}$-\emph{monic} if $$\mathcal{C}(Y,D)\xrightarrow{\mathcal{C}(f,D)} \mathcal{C}(X,D)\rightarrow 0$$ is exact for any object $D\in\mathcal{D}$. A morphism $f:X\rightarrow D$ in $\mathcal{C}$ is called a \emph{left} $\mathcal{D}$-\emph{approximation} of $X$ if $f$ is $\mathcal{D}$-monic and $D\in\mathcal{D}$. We can defined $\mathcal{D}$-\emph{epic} morphism and \emph{right} $\mathcal{D}$-\emph{approximation} dually.

\begin{definition}\label{mu}\cite[Definition 3.1]{L}
Let $\mathcal{C}$ be a $(d+2)$-angulated category with a $d$-suspension functor $\Sigma^d$, and $\mathcal{D}\subseteq\mathcal{Z}$ be a subcategory of $\mathcal{C}$. The pair $(\mathcal{Z},\mathcal{Z})$ is called a $\mathcal{D}$-$mutation\ pair$ if it satisfies the following conditions.
\vspace{1mm}

(1) For any object $X\in\mathcal{Z}$, there exists a $(d+2)$-angle $$X\xrightarrow{a_1}D_1\xrightarrow{a_2}D_2\xrightarrow{a_3}\cdots\xrightarrow{a_{d}}D_{d}\xrightarrow{a_{d+1}}Y\xrightarrow{a_{d+2}}\Sigma^d X$$
where $D_i\in\mathcal{D}, Y\in\mathcal{Z}, a_1$ is a left $\mathcal{D}$-approximation and $a_{d+1}$ is a right $\mathcal{D}$-approximation.
\vspace{1mm}

(2) For any object $Y\in\mathcal{Z}$, there exists a $(d+2)$-angle $$X\xrightarrow{c_1}D_1\xrightarrow{c_2}D_2\xrightarrow{c_3}\cdots\xrightarrow{c_{d}}D_{d}\xrightarrow{c_{d+1}}Y\xrightarrow{c_{d+2}}\Sigma^d X$$
where $X\in\mathcal{Z}, D_i\in\mathcal{D}, c_1$ is a left $\mathcal{D}$-approximation and $c_{d+1}$ is a right $\mathcal{D}$-approximation.
\end{definition}

Note that if $d=1$, then $\mathcal{C}$ is a triangulated category, and moreover if $\mathcal{D}$ is a rigid subcategory, then a $\mathcal{D}$-mutation pair is just the same as Iyama-Yoshino's definition \cite[Definition 2.5]{IY}.

For a $\mathcal{D}$-mutation pair $(\mathcal{Z},\mathcal{Z})$ in a $(d+2)$-angulated category $\mathcal{C}$, the quotient category $\mathfrak{U}:=\Z/\D$ is called a \emph{subfactor} \emph{category} which is defined by the following data.
\begin{itemize}
  \item [(1)] The objects in $\mathfrak{U}$ are the same as  $\Z$.
  \item [(2)] The morphism space $\Hom_{\mathfrak{U}}(X,Y)$ is defined as
  $$\Hom_{\mathfrak{U}}(X,Y):=\Hom_{\Z}(X,Y)/[\D](X,Y)$$
  for each $X,Y\in\Z$, where $[\D](X,Y)$ is the subspace of $\Hom_{\Z}(X,Y)$ consisting of morphisms factoring through objects in $\D$.
\end{itemize}

It is proved in \cite{L} that $\mathfrak{U}$ carries a natural $(d+2)$-angulated structure inherited from the $(d+2)$-angulated structure of $\C$ as follows.
\begin{itemize}
  \item[$\bullet$] For any object $X\in\mathfrak{U}$, choose a $(d+2)$-angle $$X\xrightarrow{a_1}D_1\xrightarrow{a_2}D_2\xrightarrow{a_3}\cdots\xrightarrow{a_{d}}D_{d}\xrightarrow{a_{d+1}}Y\xrightarrow{a_{d+2}}\Sigma^d X$$
with $D_i\in\mathcal{D}, Y\in\mathcal{Z}, a_1$ is a left $\mathcal{D}$-approximation  and $a_{d+1}$ is a right $\mathcal{D}$-approximation. Then the $d$-shift of $X$ in $\mathfrak{U}$ is defined to be $Y$, denoted by $T^dX$.
  \item[$\bullet$] For any $(d+2)$-angle  $$X_1\xrightarrow{f_1}X_2\xrightarrow{f_2}X_3\xrightarrow{f_3}\cdots\xrightarrow{f_{d+1}}X_{d+2}\xrightarrow{f_n}\Sigma^d X_1$$
in $\C$ with $X_i\in\Z$ for $i=1,\cdots,d+2$, and $f_1$ is $\D$-monic, there is a commutative diagram of $(d+2)$-angles

$$
\xymatrix{
X_1 \ar[r]^{f_1}\ar@{=}[d] & X_2 \ar[r]^{f_2}\ar[d]^{c_2} & X_3 \ar[r]^{f_3}\ar[d]^{c_3} & \cdots \ar[r]^{f_{d}} & X_{d+1} \ar[r]^{f_{d+1}}\ar[d]^{c_{d+1}}& X_{d+2} \ar[r]^{f_{d+2}}\ar[d]^{c_{d+2}} & \Sigma^d X_1 \ar@{=}[d] \\
X_1 \ar[r]^{a_1} & D_1 \ar[r]^{a_2} & D_2 \ar[r]^{a_3} & \cdots \ar[r]^{a_{d}} & D_{d} \ar[r]^{a_{d+1}}& T^dX_1 \ar[r]^{a_{d+2}} & \Sigma^d X_1. \\
}
$$
The $(d+2)$-$T^d$-sequence $$X_1\xrightarrow{\underline{f_1}}X_2\xrightarrow{\underline{f_2}}X_3\xrightarrow{\underline{f_3}}\cdots\xrightarrow{\underline{f_{d+1}}}
X_{d+2}\xrightarrow{\underline{c_{d+2}}} T^dX_1$$ is called a $standard\ (d+2)$-$angle$ in $\mathcal{Z}/\mathcal{D}$.
The class $\Phi$ is defined as the collection of
$(d+2)$-$T^d$-sequences that are isomorphic to standard
$(d+2)$-angles.
\end{itemize}

\begin{definition}
Let $\mathcal{C}$ be a $(d+2)$-angulated category with a $d$-suspension functor $\Sigma^d$.
 A subcategory $\mathcal{Z}$ is called extension closed if for any
 morphism $f_{d+2}\colon X_{d+2}\to \Sigma^d X_1$ with $X_1,X_{d+2}\in\Z$,
 there exists  a $(d+2)$-angle
$$X_1\xrightarrow{f_1}X_2\xrightarrow{f_2}X_3\xrightarrow{f_3}
\cdots\xrightarrow{f_{d}}X_{d+1}\xrightarrow{f_{d+1}}X_{d+2}\xrightarrow{f_{d+2}}\Sigma^d X_1$$ in $\mathcal{C}$ such that $X_2,X_3,\cdots, X_{d+1}\in\mathcal{Z}$.
\end{definition}
\begin{lemma} \cite[Theorem 3.7]{L}\label{lem2}
Let $\mathcal{C}$ be a $(d+2)$-angulated category with a $d$-suspension functor $\Sigma^d$, and $\mathcal{D}\subseteq\mathcal{Z}$ be subcategories of $\mathcal{C}$. If $(\mathcal{Z},\mathcal{Z})$ is a $\mathcal{D}$-mutation pair and $\mathcal{Z}$ is extension closed, then the quotient category ($\mathcal{Z}/\mathcal{D}, T^d,\Phi)$ is a $(d+2)$-angulated category.
\end{lemma}
\section{Cotorsion pairs in $(d+2)$-angulated categories}
In this section, we define cotorsion pairs and weak cotorsion pairs in $(d+2)$-angulated categories. Moreover, we give a geometric characterization of weak cotorsion pairs in $(d+2)$-angulated cluster categories of type $A$.
\vspace{1mm}

From now on, we assume that $\C:=(\C,\Sigma^d,\Theta)$  is  a $(d+2)$-angulated category.
\subsection{Cotorsion pairs}

\begin{definition}
Let $\X$ and $\Y$ be subcategories of $\C$.
\begin{itemize}
  \item [(1)] The pair $(\X, \Y)$ is called a \emph{torsion} \emph{pair}  if $\Hom_{\C}(\X, \Y)=0$ and for any $C\in\C$, there exist two $(d+2)$-angles
    $$X\rightarrow C\rightarrow Y_1\rightarrow \cdots\rightarrow Y_d\rightarrow \Sigma^d X$$
   and
    $$\Sigma^{-d} Y\rightarrow X_1\rightarrow\cdots\rightarrow X_d \rightarrow C\rightarrow Y$$
     with $X_i\in\X$, $Y_i\in\Y$ for $i=1,\cdots,d$, and $X\in\X$, $Y\in\Y$.
\vspace{1mm}

\item [(2)] The pair $(\X, \Y)$ is called a  \emph{cotorsion} \emph{pair}  if $\Hom_{\C}(\X, \Sigma^{d}\Y)=0$ and for any $C\in\C$, there exist two $(d+2)$-angles
    $$\Sigma^{-d}X\rightarrow C\rightarrow Y_1\rightarrow \cdots\rightarrow Y_d\rightarrow X$$
 and
    $$ Y\rightarrow X_1\rightarrow\cdots\rightarrow X_d \rightarrow C\rightarrow\Sigma^{d} Y$$
   with $X_i\in\X$, $Y_i\in\Y$ for $i=1,\cdots,d$, and $X\in\X$, $Y\in\Y$.
   \vspace{1mm}

 \item [(3)] The pair $(\X, \Y)$ is called a \emph{weak} \emph{cotorsion} \emph{pair} if $\X$ is contravariantly finite, $\Y$ is covariantly finite, $\Y=(\Sigma^{-d}\X){^\bot}$ and  $\X={^\bot}(\Sigma^{d}\Y)$.
\end{itemize}
\end{definition}

\begin{remark}\label{remark3}
By definition, we know that $(\X,\Y)$ is a cotorsion pair if and only if
$(\X,\Sigma^{d}\Y)$ is a torsion pair. Therefore, the concepts of cotorsion pairs and torsion pairs are equivalent in this context. This is consistent with the case in triangulated categories.
\end{remark}

We recall the following fact which is needed in the sequel.

\begin{lemma}\emph{\cite[Lemma 3.13]{F}}\label{y1}
Let $\C$ be a $(d+2)$-angulated category, and
$$A_0\xrightarrow{\alpha_0}A_1\xrightarrow{\alpha_1}A_2\xrightarrow{\alpha_2}\cdots\xrightarrow{\alpha_{d-1}}A_d\xrightarrow{\alpha_d}A_{d+1}\xrightarrow{\alpha_{d+1}}\Sigma^d A_0$$
be a $(d+2)$-angle in $\C$. Then the following statements are equivalent:
\begin{itemize}
\item[\rm (1)] $\alpha_0$ is a section;
\item[\rm (2)] $\alpha_d$ is a retraction;
\item[\rm (3)] $\alpha_{d+1}=0$.
\end{itemize}
\end{lemma}
\begin{proposition}
Let $(\X,\Y)$ be a cotorsion pair in $\C$. Then the following statements hold
\begin{itemize}
\item[\rm (1)] $\X$ is contravariantly finite;
\item[\rm (2)] $\Y$ is covariantly finite;
\item[\rm (3)] $\X={^\bot}(\Sigma^{d}\Y)$;
\item[\rm (4)] $\Y=(\Sigma^{-d}\X){^\bot}$.
\end{itemize}
\end{proposition}
\begin{proof}
We prove the statements (1) and (3), (2) and (4) can be proved similarly. Since $(\X,\Y)$ is a cotorsion pair, for arbitrary object $\Sigma^{-d}C\in\C$, there exists a $(d+2)$-angle
   $$\Sigma^{-d}X\xrightarrow{\Sigma^{-d}f} \Sigma^{-d}C\rightarrow Y_1\rightarrow \cdots\rightarrow Y_d\rightarrow X$$
with $X\in\X$ and $Y_i\in\Y$ for $i=1,\cdots,d$. Applying $\Hom(\X,-)$ to the $(d+2)$-angle above, we obtain the exact sequence
$$\Hom(\X,X)\xrightarrow{\Hom(\X,f)}\Hom(\X,C)\rightarrow 0$$
since $\Hom_{\C}(\X, \Sigma^{d}\Y)=0$. Thus $f$ is a left $\X$-approximation of $C$ and then $\X$ is contravariantly finite.

Now we give the proof of (3). Since $\Hom_{\C}(\X, \Sigma^{d}\Y)=0$, $\X\subset{^\bot}(\Sigma^{d}\Y)$. It is enough to show ${^\bot}(\Sigma^{d}\Y)\subset\X$. For any object $Z\in{^\bot}(\Sigma^{d}\Y)$, there exists a $(d+2)$-angle
   $$\Sigma^{-d}X\xrightarrow{g} \Sigma^{-d}Z\xrightarrow{h} Y_1\rightarrow \cdots\rightarrow Y_d\rightarrow X$$
with $X\in\X$ and $Y_i\in\Y$ for $i=1,\cdots,d$. Note that $h=0$, then $g$ is a retraction by Lemma \ref{y1}. So $Z$ is a direct summand of $X$, and we obtain that $\Z\in\X$.
\end{proof}

This proposition immediately yields the following conclusion.

\begin{corollary}
Let $\X$ and $\Y$ be subcategories of $\C$. If $(\X,\Y)$ is a cotorsion pair in $\C$, then $(\X,\Y)$ is a weak cotorsion pair in $\C$.
\end{corollary}

\begin{remark}
When $d=1$, $\C$ is a triangulated category, the concepts of cotorsion pair and weak cotorsion pair are concide. However, for $d\geq 2$, they are not the same in general.
\end{remark}

Before we give examples of cotorsion pairs in $\C$, we recall the following notion.

\begin{definition}\cite[Definition 0.3]{JJ1} and \cite[Definition 1.1]{ZZ1}
Let $\X$ be a subcategory of $\C$. The subcategory $\X$ is called {\em cluster tilting in $\C$} if:
\begin{enumerate}
\setlength\itemsep{4pt}

  \item $\X$ is $d$-rigid, i.e. $\Hom_{\C}(\X,\Sigma^d\X)=0$.

  \item For any \(C\in\X\) there exists a \((d+2)\)-angle
\[
  X_d\rightarrow \cdots \rightarrow X_0\rightarrow C\rightarrow \Sigma^d X_d
\]
with \(X_i\in\X\) for all \(0\leq i\leq d\).
\end{enumerate}
An object \(X\in\C\) is called {\em Oppermann--Thomas cluster tilting in $\C$} \cite[Definition.5.3]{OT} if $\add X$ is cluster tilting.
\end{definition}

\begin{lemma}
Let $\X$ be a subcategory of $\C$. Then $(\X,\X)$ is a cotorsion pair if and only if $\X$ is cluster tilting.
\end{lemma}

\begin{proof}
It is easy to check by the definitions of cotorsion pair and cluster tilting.
\end{proof}

Note that the definition of a cotorsion pair in
$\C$ is quite strict, making it challenging to find cotorsion pairs in
$\C$. However, there are many examples of weak cotorsion pairs, as shown in the following examples.

\begin{example}
Let $d = 3$ and $\cT = \cO_{ A_2^3 }$.  This is the $5$-angulated (higher) cluster category of type $A_2$, see \cite{OT,JJ}.  The indecomposable objects can be identified with the elements of the set
\[
  {}^{ \circlearrowleft }\mbox{\bf I}^3_9
  =
  \{\, 1357,1358,1368,1468,2468,2469,2479,2579,3579 \,\},
\]
see \cite[Section \ 8]{OT}.  The AR quiver of $\cT$ is shown in Figure \ref{fig:AR_quiver}.
\begin{figure}[h]
\begin{tikzpicture}[scale=3]
  \node at (0:1.0){$1357$};
  \draw[->] (10:1.0) arc (10:30:1.0);
  \node at (40:1.0){$1358$};
  \draw[->] (50:1.0) arc (50:68:1.0);
  \node at (80:1.0){$1368$};
  \draw[->] (91:1.0) arc (91:108:1.0);
  \node at (120:1.0){$1468$};
  \draw[->] (131:1.0) arc (131:150:1.0);
  \node at (160:1.0){$2468$};
  \draw[->] (170:1.0) arc (170:190:1.0);
  \node at (200:1.0){$2469$};
  \draw[->] (210:1.0) arc (210:227:1.0);
  \node at (240:1.0){$2479$};
  \draw[->] (251:1.0) arc (251:268:1.0);
  \node at (280:1.0){$2579$};
  \draw[->] (295:1.0) arc (295:310:1.0);
  \node at (320:1.0){$3579$};
  \draw[->] (330:1.0) arc (330:351:1.0);
\end{tikzpicture}
\caption{The AR quiver of the $5$-angulated category $\cT$}
\label{fig:AR_quiver}
\end{figure}
By \cite[Theorem 5.5 and Section 8]{OT}, the object
\[
  T = 1357 \oplus 1358 \oplus 1368 \oplus 1468
\]
is Oppermann--Thomas cluster tilting.  So $(\add T,\add T)$ is a cotorsion pair in $\cT$. Moreover, the object
\[
  M = 1357 \oplus 1468 \oplus 2479
\]
is a maximal 3-rigid objects of $\cT$ by \cite[Section 4]{JJ}, but not cluster tilting, so $(\add M,\add M)$ is a weak cotorsion pair, but not a cotorsion pair in $\cT$.
\end{example}
\begin{example}
Let $\T=D^b(kQ)/\tau^{-1}[1]$ be the cluster category of type $A_3$, where $Q$ is the quiver
$1\rightarrow 2\rightarrow 3$, $D^b(kQ)$ is the bounded derived category of finite generated modules over $kQ$, see \cite[Example 4.1]{L}. Then $\T$ is a triangulated category. The AR-quiver of $\T$ is given by Figure \ref{2}.
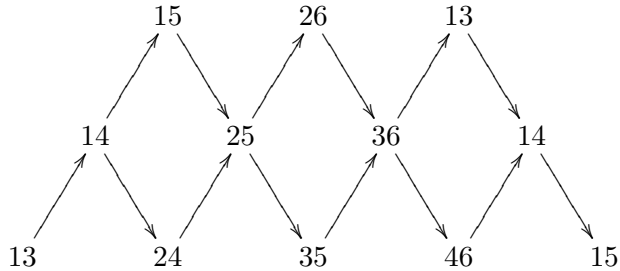
\begin{figure}[h]
\centering

$$
{
\xymatrix@-4mm@C-0.1cm{
     &&15  \ar[rdd] & & 26  \ar[rdd]&& 13 \ar[rdd]\\
 \\
 & 14 \ar[rdd] \ar[ruu] & & 25 \ar[rdd]\ar[ruu] & & 36 \ar[rdd] \ar[ruu] & & 14 \ar[rdd]\\
 \\
  13 \ar[ruu]&& 24 \ar[ruu] && 35\ar[ruu]&&46 \ar[ruu]&& 15 \\
}
}
$$
\caption{The AR-quiver of $\T$}
\label{2}
\end{figure}

Let $\C=\add(13\oplus15\oplus35)$. Then one can check $\C$ is a 2-cluster tilting subcategory of $\T$ and $\C[2]=\C$. Thus $(\C,[2])$ is a 4-angulated category. The AR quiver of $\C$ is shown in Figure \ref{cc}.
\begin{figure}[h]
\begin{tikzpicture}[scale=3]
  \node at (0:1.0){$13$};
  \draw[->] (10:1.0) arc (10:110:1.0);
  \node at (120:1.0){$15$};
  \draw[->] (130:1.0) arc (130:230:1.0);
  \node at (240:1.0){$35$};
  \draw[->] (250:1.0) arc (250:350:1.0);
\end{tikzpicture}
\caption{The AR quiver of the 4-angulated category $\C$}
\label{cc}
\end{figure}
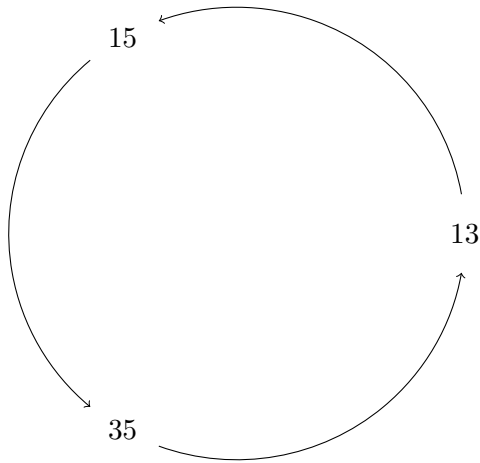
The cotorsion pairs we can find are only  $(\C, 0)$ and $(0, \C)$.
\end{example}

\subsection{Weak cotorsion pairs in  $(d+2)$-angulated cluster categories of type $A$}
In this subsection, we give a geometric characterization of weak cotorsion pairs in $(d+2)$-angulated cluster categories of type $A$, which was introduced in \cite[Section 6]{OT}. First, we briefly review some definitions and related results.

Iyama introduced higher dimensional analogues of Auslander-Reiten theory and Auslander algebras (see, for instance \cite{Iya}), generalizing several classical concepts from the representation theory of finite dimensional algebras. In \cite{IO}, the notion of $d$-representation finiteness was introduced as an ideal framework for studying these concepts. (See also \cite{Iya, IO2, HI})

In this subsection, we focus on the most well-understood class of $d$-representation finite algebras: the $(d-1)$-th higher Auslander algebras of linearly oriented $A_n$ (see \cite{Iya,IO}), which are denoted by $A_n^d$. The case when $d=1$ is already well understood. In this case, the algebra $A_n^1$ is the path algebra of linearly ordered $A_n$.

We recall a higher dimensional cluster category, denoted $\mathscr{O}_{\Lambda}$,  for any $d$-representation finite algebra $\Lambda$, as introduced in \cite{OT}.
 The cluster category $\mathscr{O}_{\Lambda}$ is constructed as a subcategory of the $2d$-Amiot cluster category $\mathscr{C}_{\Lambda}^{2d}$ in \cite{A}. The $2d$-Amiot cluster categories generalize classical $2d$-cluster categories to non-hereditary algebras $\Lambda$. In particular, the categories $\mathscr{C}_{\Lambda}^{2d}$ are $2d$-Calabi-Yau and triangulated. These properties of $\mathscr{C}_{\Lambda}^{2d}$ are used to show that $\mathscr{O}_{\Lambda}$ is $(d+2)$-angulated, and also satisfies a certain Calabi-Yau property. It is worth noting that for $d=1$, the categories $\mathscr{O}_{\Lambda}$ and $\mathscr{C}_{\Lambda}^2$, and both coincide with the classical cluster category of the hereditary representation finite algebra $\Lambda$. For all $d$-representation finite algebras $\Lambda$, the category $\mathscr{O}_{\Lambda}$ contains only finitely many indecomposable objects, which can be arranged in a $d$-dimensional analogue of an Auslander-Reiten quiver.

For the case $\Lambda = A_n^d$, we obtain the $(d+2)$-angulated cluster categories of type $A$, denoted by $\mathscr{O}_{A_n^d}$. The following property is easy to verify based on the construction of $\mathscr{O}_{A_n^d}$.

\begin{lemma}
Let $\X,\Y$ be two subcategories of $\mathscr{O}_{A_n^d}$. Then we have an isomorphism
$$\Hom_{\mathscr{O}_{A_n^d}}(\X,\Sigma^{d}\Y)\cong\Hom_{\mathscr{O}_{A_n^d}}(\Y,\Sigma^{d}\X)$$
\end{lemma}

We recall the geometric description of $\mathscr{O}_{A_n^d}$ based on \cite[Section\ 6]{OT}. For general positive integers $m,r$, index sets are defined as follows:
\begin{align*}
  \mathbf{I}_m^r & = \{ (i_0, \ldots, i_r) \in \{1, \ldots, m\}^{r+1} \mid \forall x
\in \{0,1,\dots,r-1\} \colon i_x + 2 \leq i_{x+1}\}\\
\leftsup{\circlearrowleft}{\mathbf{I}_m^r} &= \{ (i_0, \ldots, i_r) \in \mathbf{I}_m^r  \mid  i_r + 2 \leq i_0 + m\}.
\end{align*}

Let $X$ and $Y$ be increasing
$(r+1)$-tuples of real numbers which belong to $\leftsup{\circlearrowleft}{\mathbf{I}_m^r}$.  We say that $X=(x_0,\dots,x_{r})$ \emph{intertwines}
$Y=(y_0,\dots,y_{r})$  if $x_0<y_0<x_1<y_1\dots<x_{r}<y_{r}$.
We write $X \wr Y$ for this relation.  A collection of increasing $(d+1)$-tuples is called \emph{non-intertwining} if
no pair of the elements intertwine (in either order).

We call the elements in $\leftsup{\circlearrowleft}{\mathbf{I}_{n+2d+1}^d}$ \emph{diagonals}. There is a bijection between indecomposable objects in $\mathscr{O}_{A_n^d}$ and diagonals in $\leftsup{\circlearrowleft}{\mathbf{I}_{n+2d+1}^d}$. For convenience, we index the indecomposable objects in $\mathscr{O}_{A_n^d}$ by $\leftsup{\circlearrowleft}{\mathbf{I}_{n+2d+1}^d}$ as in \cite{OT}. For a set of diagonals $\mathfrak{X}$ in $\leftsup{\circlearrowleft}{\mathbf{I}_{n+2d+1}^d}$,  we denote the corresponding subcategory of $\mathscr{O}_{A_n^d}$ by $\X$. We write $O_{i_0, \ldots, i_d}$ for the indecomposable object corresponding to $(i_0, \ldots, i_d) \in \leftsup{\circlearrowleft}{\mathbf{I}_{n+2d+1}^d}$, sometimes $(i_0, \ldots, i_d)$ represents an indecomposable object in $\mathscr{O}_{A_n^d}$ directly without confusion.

The following lemma is crucial for our description of weak cotorsion pairs in $\mathscr{O}_{A_n^d}$.
\begin{lemma} \cite[Proposition 6.1]{OT}\label{lem}
Let $(i_0, \ldots, i_d), (j_0, \ldots, j_d) \in \leftsup{\circlearrowleft}{\mathbf{I}_{n+2d+1}^d}$. Then
\[ \Hom_{\mathscr{O}_{A_n^d}}(O_{i_0, \ldots, i_d}, O_{j_0, \ldots, j_d}[d]) \neq 0 \iff \big[ (i_0, \ldots, i_d) \wr (j_0, \ldots, j_d) \text{ or } (j_0, \ldots, j_d) \wr (i_0, \ldots, i_d) \big], \]
and in this case the $\Hom$-space is one-dimensional.
\end{lemma}
Let  $\mathfrak{X}$ be a set of diagonals in $\leftsup{\circlearrowleft}{\mathbf{I}_{n+2d+1}^d}$. We define
$$\nc\mathfrak{X}=\{u\in\leftsup{\circlearrowleft}{\mathbf{I}_{n+2d+1}^d}|\;u \text{\;does\;not\;intertwines\;any\;} \text{diagonals\; in\;}\mathfrak{X}\}.$$
\begin{remark}\label{rem}
By Lemma \ref{lem}, $\nc\mathfrak{X}$ corresponds to the subcategory $(\Sigma^{-d}\X)^\perp={^\bot}(\Sigma^{d}\X)$.
\end{remark}
\begin{theorem}\label{i}
 Let $\X,\Y$ be subcategories of $\mathscr{O}_{A_n^d}$, and $\mathfrak{X},\mathfrak{Y}$ be the corresponding set of diagonals in $\leftsup{\circlearrowleft}{\mathbf{I}_{n+2d+1}^d}$ respectively. Then the following statements are equivalent.
\begin{itemize}
  \item [(1)] $(\X,\Y)$ is a weak cotorsion pair.
  \item [(2)] $\mathfrak{X}=\nc\mathfrak{Y}$ and $\mathfrak{Y}=\nc\mathfrak{X}$.
  \item [(3)] $(\Y,\X)$ is a weak cotorsion pair.
  \item [(4)] $\mathfrak{Y}=\nc\mathfrak{X}$ and $\mathfrak{X}=\nc\nc\mathfrak{X}$.
  \item [(5)] $\mathfrak{X}=\nc\mathfrak{Y}$ and $\mathfrak{Y}=\nc\nc\mathfrak{Y}$.
  \end{itemize}
\end{theorem}
\begin{proof}
Note that the category $\mathscr{O}_{A_n^d}$ contains only finitely many indecomposable objects, so every subcategory of $\mathscr{O}_{A_n^d}$ is functorially finite.

``$(1)\Rightarrow(2)$": Suppose $(\X,\Y)$ is a weak cotorsion pair, then $\X={^\bot}(\Sigma^{d}\Y)$ and  $\Y=(\Sigma^{-d}\X){^\bot}$. By Remark \ref{rem}, we have $\mathfrak{X}=\nc\mathfrak{Y}$ and $\mathfrak{Y}=\nc\mathfrak{X}$.

``$(2)\Rightarrow(3)$": It is easy to check by Remark \ref{rem}.

``$(3)\Rightarrow(4)$": Since $(\Y,\X)$ is a weak cotorsion pair, we have $\Y={^\bot}(\Sigma^{d}\X)$ and  $\X=(\Sigma^{-d}\Y){^\bot}$. This means $\mathfrak{Y}=\nc\mathfrak{X}$ and $\mathfrak{X}=\nc\nc\mathfrak{X}$ hold.

``$(4)\Rightarrow(5)$": One can check directly.

``$(5)\Rightarrow(1)$": $\mathfrak{X}=\nc\mathfrak{Y}$ and $\mathfrak{Y}=\nc\nc\mathfrak{Y}$ imply $\mathfrak{X}=\nc\mathfrak{Y}$ and $\mathfrak{Y}=\nc\mathfrak{X}$, which means $\X={^\bot}(\Sigma^{d}\Y)$ and  $\Y=(\Sigma^{-d}\X){^\bot}$ by Remark \ref{rem}, so $(\X,\Y)$ is a weak cotorsion pair.
\end{proof}
When $d=1$, $\mathscr{O}_{A_n^1}$ is the classical cluster categories of type $A_n$ and  $\leftsup{\circlearrowleft}{\mathbf{I}_{n+3}^1}$ is just the  $(n+3)$-gon geometric model given by \cite{HJR1}. Moreover, for a set of diagonals $\mathfrak{X}$ in $\leftsup{\circlearrowleft}{\mathbf{I}_{n+3}^1}$, our definition of  $\nc\mathfrak{X}$ is compatible with \cite{HJR1}, so we have the following corollary.
\begin{corollary}\cite[Proposition 2.3]{HJR1}
Let $\mathscr{O}_{A_n^1}$ be the cluster categories of type $A_n$. Then the following statements are equivalent.
\begin{itemize}
  \item [(1)] $(\X,\X{^\bot})$ is a torsion pair.
  \item [(2)] $\mathfrak{X}=\nc\nc\mathfrak{X}$.
  \end{itemize}
\end{corollary}
\begin{example}
Let $d = 2$ and $\cT = \cO_{ A_2^2 }$.  This is the $4$-angulated (higher) cluster category of type $A_2$.  The indecomposable objects can be identified with the elements of the set
\[
  {}^{ \circlearrowleft }\mbox{\bf I}^2_7
  =
  \{\, 135,136,146,246,247,257,357 \,\}.
\]
The AR quiver of $\cT$ is shown in Figure \ref{fig:AR_quiver1}.
\begin{figure}[h]
\begin{tikzpicture}[scale=3]
  \node at (0:1.0){$135$};
  \draw[->] (10:1.0) arc (10:40:1.0);
  \node at (50:1.0){$136$};
  \draw[->] (60:1.0) arc (60:90:1.0);
  \node at (100:1.0){$146$};
  \draw[->] (110:1.0) arc (110:140:1.0);
  \node at (150:1.0){$246$};
  \draw[->] (160:1.0) arc (160:190:1.0);
  \node at (200:1.0){$247$};
  \draw[->] (210:1.0) arc (210:240:1.0);
  \node at (250:1.0){$257$};
  \draw[->] (260:1.0) arc (260:290:1.0);
  \node at (300:1.0){$357$};
  \draw[->] (310:1.0) arc (310:350:1.0);
\end{tikzpicture}
\caption{The AR quiver of the $4$-angulated category $\cT$}
\label{fig:AR_quiver1}
\end{figure}

If $X,Y \in \cT$ are indecomposable objects, then by Lemma \ref{lem} we have
\[
  \cT( X,Y )
  =
  \left\{
    \begin{array}{cl}
      k & \mbox{ if $Y$ is $X$ or its immediate successor in the AR quiver, }\\
      0 & \mbox{ otherwise. }
    \end{array}
  \right.
\]
For example, the functor $\Hom_{ \cT }( 135,\Sigma^2(-) )$ is non-zero on $246$ and $247$.  It is zero on every other indecomposable object. Now we give a complete classification of weak cotorsion pairs  in $\cT$.
\begin{eqnarray*}
\X_{1}=\{ 0\}&\quad\quad& \Y_{1}=\{\cT \}\\
\X_{2}=\{135 \}&\quad\quad& \Y_{2}=\{135,136,146,257,357 \}\\
\X_{3}=\{136 \}&\quad\quad& \Y_{3}=\{135,136,146,246,357 \}\\
\X_{4}=\{146 \}&\quad\quad& \Y_{4}=\{135,136,146,246,247 \}\\
\X_{5}=\{246 \}&\quad\quad& \Y_{5}=\{136,146,246,247,257 \}\\
\X_{6}=\{247 \}&\quad\quad& \Y_{6}=\{146,246,247,257,357 \}\\
\X_{7}=\{257 \}&\quad\quad& \Y_{7}=\{246,247,257,357,135 \}\\
\X_{8}=\{357 \}&\quad\quad& \Y_{8}=\{247,257,357,135,136 \}\\
\X_{9}=\{135,136 \}&\quad\quad& \Y_{9}=\{135,136,146,357 \}\\
\X_{10}=\{136, 146 \}&\quad\quad& \Y_{10}=\{135,136,146,246 \}\\
\X_{11}=\{146,246 \}&\quad\quad& \Y_{11}=\{136,146,246,247 \}\\
\X_{12}=\{246,247 \}&\quad\quad& \Y_{12}=\{146,246,247,257 \}\\
\X_{13}=\{247,257 \}&\quad\quad& \Y_{13}=\{246,247,257,357 \}\\
\X_{14}=\{257,357 \}&\quad\quad& \Y_{14}=\{247,257,357,135 \}\\
\X_{15}=\{357,135 \}&\quad\quad& \Y_{15}=\{257,357,135,136 \}\\
\X_{16}=\{135,246 \}&\quad\quad& \Y_{16}=\{136,146,257 \}\\
\X_{17}=\{135,247 \}&\quad\quad& \Y_{17}=\{146,257,357 \}\\
\X_{18}=\{136,247 \}&\quad\quad& \Y_{18}=\{146,246,357 \}\\
\X_{19}=\{136,257 \}&\quad\quad& \Y_{19}=\{135,246,357 \}\\
\X_{20}=\{146,257 \}&\quad\quad& \Y_{20}=\{135,246,247 \}\\
\X_{21}=\{146,357 \}&\quad\quad& \Y_{21}=\{135,136,247 \}\\
\X_{22}=\{246,357 \}&\quad\quad& \Y_{22}=\{136,247,257 \}\\
&\X_{23}=\{135,136,146 \}=\Y_{23}\\
&\X_{24}=\{136,146,246 \}=\Y_{24}\\
&\X_{25}=\{146,246,247 \}=\Y_{25}\\
&\X_{26}=\{246,247,257 \}=\Y_{26}\\
&\X_{27}=\{247,257,357 \}=\Y_{27}\\
&\X_{28}=\{257,357,135 \}=\Y_{28}\\
&\X_{29}=\{357,135,136 \}=\Y_{29}\\
\end{eqnarray*}
Note that both $(\X_i,\Y_i)$ and $(\Y_i,\X_i)$ are weak cotorsion pairs in $\cT$ for $i=1,2,\cdots,22$. When $i=23,\cdots,29$, $\X_i=\Y_i$ are cluster tilting subcategories of $\cT$, $(\X_i,\Y_i)=(\Y_i,\X_i)$ are weak cotorsion pairs.
\end{example}

\section{Mutation of cotorsion pairs in $(d+2)$-angulated categories}
In this section, we define mutation of cotorsion pairs in $(d+2)$-angulated categories, and show that any mutation of a (weak) cotorsion pair in $\C$ is again a (weak) cotorsion pair.

\subsection{Mutation of cotorsion pairs}
First, we give the definition of mutation of subcategories,
analogous to the mutation pairs defined in Definition \ref{mu}.

\begin{definition}\label{mu}
Let $X$ be an object of $\mathcal{C}$.
\begin{itemize}
\item [(1)] We call $\mu^{-1}_{\D}(X)$ the \emph{forward} $\mathcal{D}$-\emph{mutation} of $X$ if there exists a $(d+2)$-angle $$X\xrightarrow{a_1}D_1\xrightarrow{a_2}D_2\xrightarrow{a_3}\cdots\xrightarrow{a_{d}}D_{d}\xrightarrow{a_{d+1}}\mu^{-1}_{\D}(X)\xrightarrow{a_{d+2}}\Sigma^d X$$
where $D_i\in\mathcal{D}$ for $i=1,\cdots,d$, $a_1$ is a left $\mathcal{D}$-approximation and $a_{d+1}$ is a right $\mathcal{D}$-approximation.
\vspace{1mm}

\item [(2)] We call $\mu_{\D}(X)$ the \emph{backward} $\mathcal{D}$-\emph{mutation} of $X$ if there exists a $(d+2)$-angle
 $$\mu_{\D}(X)\xrightarrow{c_1}D_1\xrightarrow{c_2}D_2\xrightarrow{c_3}\cdots\xrightarrow{c_{d}}D_{d}\xrightarrow{c_{d+1}}X\xrightarrow{c_{d+2}}\Sigma^d X$$
where $D_i\in\mathcal{D}$ for $i=1,\cdots,d$, $c_1$ is a left $\mathcal{D}$-approximation and $c_{d+1}$ is a right $\mathcal{D}$-approximation.
\end{itemize}
\end{definition}
\begin{definition}
Given two subcategories $\X,\D$ of $\mathcal{C}$ with $\D\subset\X$.
\begin{itemize}
\item [(1)] We define the \emph{forward} $\mathcal{D}$-\emph{mutation} of $\X$ by
$$\mu^{-1}_{\D}(\X):=\{\mu^{-1}_{\D}(X)\mid X\in\X\}\cup\D.$$
\item [(2)] We define the \emph{backward} $\mathcal{D}$-\emph{mutation} of $\X$ by
$$\mu_{\D}(\X):=\{\mu_{\D}(X)\mid X\in\X\}\cup\D.$$
\end{itemize}
\end{definition}

\begin{remark}
It is clear that $\mu_{\D}(\D)=\D=\mu^{-1}_{\D}(\D)$. When $\D=0$, we have $\mu^{-1}_{\D}(\X)=\Sigma^d\X$ and $\mu_{\D}(\X)=\Sigma^{-d}\X$.
\end{remark}

We give some definitions which will be used in what follows.

\begin{definition}
Let $(\X,\Y)$ be a cotorsion pair in $\C$. We call $\I(\X)=\X\cap\Y$ the core of $(\X,\Y)$.
\end{definition}

\begin{remark}\label{l}
By the definition of cotorsion pair, it is easy to see the core $\I(\X)$ is a $d$-rigid subcategory.
\end{remark}

\begin{definition}\cite[Definition 5.1]{Z}\label{def1}
Let $\D$ be a subcategory of $\C$. The subcategory $\D$ is called \emph{strongly} \emph{contravariantly} \emph{finite}, if for any object $C\in\C$, there exists a $(d+2)$-angle
\begin{equation}\label{t1}
\begin{array}{l}
B\rightarrow D_1\rightarrow D_2\rightarrow \cdots\rightarrow D_{d}\xrightarrow{g}C\rightarrow \Sigma^d B,
\end{array}
\end{equation}
where $D_i\in\mathcal{D}$ for $i=1,\cdots,d$, $g$ is a right $\mathcal{D}$-approximation. Dually,  $\D$ is called \emph{strongly} \emph{covariantly} \emph{finite}, if for any object $C\in\C$, there exists a $(d+2)$-angle
$$C\xrightarrow{f} D_1\rightarrow D_2\rightarrow \cdots\rightarrow D_{d}\rightarrow A\rightarrow \Sigma^d C,$$
where $D_i\in\mathcal{D}$ for $i=1,\cdots,d$, $f$ is a left $\mathcal{D}$-approximation.

A strongly contravariantly finite and strongly covariantly finite subcategory is called \emph{strongly} \emph{functorially} \emph{finite}.
\end{definition}

\begin{remark}\label{rem1}
In Definition \ref{def1}, if we further assume that $\D$ is a $d$-rigid subcategory of $\C$, thus we can obtain $B\in(\Sigma^{-d}\D){^\bot}$ and $A\in{^\bot}(\Sigma^{d}\D)$.
In fact, applying the functor $\C(\D,-)$ to the $(d+2)$-angle (\ref{t1}),
we have the following exact sequence:
$$\C(\D,D_d)\xrightarrow{\C(\D,~g)}\C(\D,C)\xrightarrow{~}\C(\D,\Sigma^d B)\xrightarrow{~}\C(\D,D_1)=0.$$
Since $g$ is a right $\D$-approximation of $C$,
we have that $\C(\D,g)$ is an epimorphism. It follows that
$\C(\D,\Sigma^d B)=0$ and then $B\in(\Sigma^{-d}\D){^\bot}$.
Similarly, we can show that $A\in{^\bot}(\Sigma^{d}\D)$.
\end{remark}

\subsection{Compatibility with subfactor $(d+2)$-angulated categories}
In this subsection, we fix a strongly functorially finite $d$-rigid subcategory $\D$ of $\C$ which satisfies the condition ${^\bot}(\Sigma^d\D)=(\Sigma^{-d}\D){^\bot}$, which is denoted by $\Z$. We can show that $(\Z,\Z)$ is a $\mathcal{D}$-mutation pair.

\begin{lemma}\label{lem3}
Let $\D$ be a strongly functorially finite $d$-rigid subcategory of $\C$ which satisfies ${^\bot}(\Sigma^d\D)=(\Sigma^{-d}\D){^\bot}$. Then $(\Z,\Z)$ is a $\mathcal{D}$-mutation pair.
\end{lemma}

\begin{proof}
Since $\D$ is a strongly covariantly finite $d$-rigid subcategory of $\C$, we have $\D\subset\Z$ and for any object $Z\in\Z$, there exists a $(d+2)$-angle
$$Z\xrightarrow{f} D_1\rightarrow D_2\rightarrow \cdots\rightarrow D_{d}\xrightarrow{g} A\rightarrow\Sigma^d Z,$$
where $D_i\in\mathcal{D}$ for $i=1,\cdots,d$, $f$ is a left $\mathcal{D}$-approximation.
By Remark \ref{rem1}, we have $A\in{^\bot}(\Sigma^{d}\D)$.
 By assumption $A\in{^\bot}(\Sigma^{d}\D)=(\Sigma^{-d}\D){^\bot}=\Z$.
 Since $\Hom_{\C}(\D,\Sigma^d Z)=0$, we have that $g$ is a right $\mathcal{D}$-approximation of $A$.

 Since $\D$ is a strongly contravariantly finite subcategory of $\C$, for any object $Z\in\Z$, we can prove similarly that there exists a $(d+2)$-angle
$$B\xrightarrow{h} D_1\rightarrow D_2\rightarrow \cdots\rightarrow D_{d}\xrightarrow{k}Z\rightarrow \Sigma^d B,$$
where $D_i\in\mathcal{D}$ for $i=1,\cdots,d$, $h$ is a left $\mathcal{D}$-approximation, $k$ is a right $\mathcal{D}$-approximation and $B\in\Z$. This shows that $(\Z,\Z)$ is a $\mathcal{D}$-mutation pair.
\end{proof}

The following results are useful in the sequel.

\begin{lemma}\label{c}
Let $\D$ be a strongly functorially finite $d$-rigid subcategory of $\C$  satisfying  ${^\bot}(\Sigma^d\D)=(\Sigma^{-d}\D){^\bot}$, which is denoted by $\Z$. Moreover, we assume that $\Z$ is extension closed. Then we have the following results.
\begin{itemize}
  \item [(1)]  The subfactor category $\mathfrak{U}:=\Z/\D$ is a $(d+2)$-angulated category.
  \vspace{1mm}

  \item [(2)] For any $X,Y\in\Z$, there exists an isomorphism
  $$\mathfrak{U}(X,T^d Y)\cong\C(X,\Sigma^d Y).$$
\end{itemize}
\end{lemma}
\begin{proof}
The statement (1) is obvious by Lemma \ref{lem2}. We only show the statement (2).

For any $Y\in\Z$, since $(\Z,\Z)$ is a $\mathcal{D}$-mutation pair by Lemma \ref{lem3}, there exists a $(d+2)$-angle
$$Y\xrightarrow{f} D_1\rightarrow D_2\rightarrow \cdots\rightarrow D_{d}\xrightarrow{g} T^d Y\rightarrow\Sigma^d Y,$$
where $D_i\in\mathcal{D}$ for $i=1,\cdots,d$, $T^d Y\in\Z$, $f$ is a left $\mathcal{D}$-approximation and $g$ is a right $\mathcal{D}$-approximation.

For any $X\in\Z$, applying $\Hom_{\C}(X,-)$ to the above $(d+2)$-angle, we have the following exact sequence:
$$\C(X,D_{d})\xrightarrow{\C(X,~g)}\C(X,T^d Y)\rightarrow\C(X,\Sigma^d Y)\rightarrow\C(X,\Sigma^d D_1).$$
Since $\C(\Z,\Sigma^d\D)=0$, we have $\C(X,\Sigma^d D_1)=0$. In this case, we get
$$\C(X,\Sigma^d Y)\cong\C(X,T^d Y)/\text{Im}(C(X,g)).$$
 It is enough to show $\text{Im}(\C(X,g))=[\D](X,T^d Y)$.
For any morphism $h\in\text{Im}(\C(X,g))$, there exists a morphism $k\in\C(X,D_{d})$ such that $h=g\circ k$. It follows that $h\in[\D](X,T^d Y)$. Conversely, for any morphism $s\in[\D](X,T^d Y)$, there exists an object $D^\prime\in\D$ such that $s=a\circ b$, where $a :D^\prime\rightarrow T^d Y$ and $b:X\rightarrow D^\prime$. Since $g: D_{d}\rightarrow T^d Y$ is a right $\mathcal{D}$-approximation, there exists a morphism $t:D^\prime\rightarrow D_d$ such that $a=g\circ t$. It follows that
$$s=a\circ b=s=g\circ t\circ b\in\text{Im}(\Hom_{\C}(X,g)).$$

Since $X,T^d Y\in\Z$, we have $\C(X,T^d Y)=\Z(X,T^d Y)$. Hence $$\mathfrak{U}(X,T^d Y)=\Z(X,T^d Y)/[\D](X,T^d Y)=\C(X,T^d Y)/\text{Im}(C(X,g))\cong\C(X,\Sigma^d Y).$$
\end{proof}

\begin{lemma}\label{b}
Let $(\X,\Y)$ be a cotorsion pair in $\C$ with core $\I(\X)$. Then $\D\subset\I(\X)$ if and only if $\D\subset\X\subset\Z$ and $\D\subset\Y\subset\Z$.
\end{lemma}

\begin{proof}
The `if' part is trivial, since  $\I(\X)=\X\cap\Y$. To prove the `only if' part, let $X$ be an arbitrary object in $\X$. Since  $\D\subset\I(\X)\subset\Y$, we have $\Hom(X,\Sigma^d\D)=0$. It follows that $X\in\Z$ and hence $\X\subset\Z$. Similarly, we have $\Y\subset\Z$.
\end{proof}
From now on, we assume that $\Z$ is extension closed. By the lemma above, one can consider the relationship between cotorsion pairs in $\C$ whose cores contain $\D$ and
cotorsion pairs in the subfactor $(d+2)$-angulated category $\mathfrak{U}$. In the following, for a subcategory $\W$ of $\C$ satisfying  $\D\subset\W\subset\Z$, we denote $\overline{\W}$ the subcategory of $\mathfrak{U}$. It is clear that any subcategory of $\mathfrak{U}$ has this form.
\begin{theorem}\label{d}
Let $(\X,\Y)$ be a cotorsion pair in $\C$ with $\D\subset\I(\X)$. Then $(\overline{\X},\overline{\Y})$ is a cotorsion pair in $\mathfrak{U}$ with $\I(\overline{\X})=\overline{\I(\X)}$. Moreover, the map $(\X,\Y)\mapsto(\overline{\X},\overline{\Y})$ is a bijection from the set of cotorsion pair in $\C$ whose cores contain $\D$ to the set of cotorsion pair in $\mathfrak{U}$.
\end{theorem}
\begin{proof}
Since $(\X,\Y)$ is a cotorsion pair in $\C$ with $\D\subset\I(\X)$, we have $\D\subset\X\subset\Z$ and $\D\subset\Y\subset\Z$ by Lemma \ref{b}. So $\overline{\X}$ and $\overline{\Y}$ are subcategories of $\mathfrak{U}$. By Lemma \ref{c}, there exists an isomorphism $\mathfrak{U}(\overline{\X},T^d \overline{\Y})\cong\C(\X,\Sigma^d \Y)$, so $\mathfrak{U}(\overline{\X},T^d \overline{\Y})=0$. Since $(\X,\Y)$ is a cotorsion pair in $\C$, for any object $C\in\Z$, there exist two $(d+2)$-angles
    $$\Sigma^{-d}X\rightarrow C\rightarrow Y_1\rightarrow \cdots\rightarrow Y_d\rightarrow X$$
 and
    $$ Y\rightarrow X_1\rightarrow\cdots\rightarrow X_d \rightarrow C\rightarrow\Sigma^{d} Y$$
   with $X_i\in\X$, $Y_i\in\Y$ for $i=1,\cdots,d$, and $X\in\X$, $Y\in\Y$.
By the construction of standard $(d+2)$-angles in $\mathfrak{U}$, we have two $(d+2)$-angles in $\mathfrak{U}$
    $$\T^{-d}X\rightarrow C\rightarrow Y_1\rightarrow \cdots\rightarrow Y_d\rightarrow X$$
 and
    $$ Y\rightarrow X_1\rightarrow\cdots\rightarrow X_d \rightarrow C\rightarrow T^{d} Y$$
   with $X_i\in\overline{\X}$, $Y_i\in\overline{\Y}$ for $i=1,\cdots,d$, and $X\in\overline{\X}$, $Y\in\overline{\Y}$. Hence, $(\overline{\X},\overline{\Y})$ is a cotorsion pair in $\mathfrak{U}$. In this case, we have  $\I(\overline{\X})=\overline{\X}\cap\overline{\Y}=\overline{\X\cap\Y}=\overline{\I(\X)}$.
\vspace{2mm}

Let $(\X_1,\Y_1),(\X_2,\Y_2)$ be two cotorsion pairs in $\C$ with $\D\subset\I(\X_1),\D\subset\I(\X_2)$. If $\overline{\X_1}=\overline{\X_2}$, then by definition $\X_1=\X_2$. Hence the map $(\X,\Y)\mapsto(\overline{\X},\overline{\Y})$ is injective. Now suppose $\X$ and $\Y$ are two subcategories of $\C$ satisfying $\D\subset\X\subset\Z$ and $\D\subset\Y\subset\Z$ such that $(\overline{\X},\overline{\Y})$ is a cotorsion pair in $\mathfrak{U}$. Then $\mathfrak{U}(\overline{\X},T^d \overline{\Y})=0$, which implies $\C(\X,\Sigma^d \Y)=0$ by Lemma \ref{c}. For any object $C\in\C$, $\Sigma^d C\in\C$, since $\D$ is strongly contravariantly finite, there exists a $(d+2)$-angle $$B\rightarrow D_1\rightarrow D_2\rightarrow \cdots\rightarrow D_{d}\xrightarrow{g}\Sigma^d C\rightarrow \Sigma^d B,$$
where $D_i\in\mathcal{D}$ for $i=1,\cdots,d$, $g$ is a right $\mathcal{D}$-approximation and $B\in\Z$. For the object $B$, Since $(\overline{\X},\overline{\Y})$ is a cotorsion pair in $\mathfrak{U}$, there exists a $(d+2)$-angle
    $$T^{-d}X\rightarrow B\rightarrow Y_1\rightarrow \cdots\rightarrow Y_d\rightarrow X$$
with $X\in\overline{\X}$, $Y_i\in\overline{\Y}$ for $i=1,\cdots,d$. This implies a $(d+2)$-angle in $\C$
$$\Sigma^{-d}X\rightarrow B\rightarrow Y_1\rightarrow \cdots\rightarrow Y_d\rightarrow X$$
with $X\in\X$, $Y_i\in\Y$ for $i=1,\cdots,d$.
Thus, we have a commutative diagram of $(d+2)$-angles in $\C$
$$\xymatrix{
B \ar[r]\ar@{=}[d] & Y_1 \ar[r]\ar@{->}[d]^{\varphi} & \cdots \ar[r] &Y_d \ar[r]& X\ar[r] & \Sigma^d B \ar@{=}[d]\\
B \ar[r] & D_1 \ar[r] &\cdots \ar[r] &D_d \ar[r]^{g} & \Sigma^dC \ar[r]& \Sigma^d B\\
}$$
Since $\C(\Sigma^{-d}\X, \D)=0$, the morphism $\varphi$ exists. By Lemma \ref{thm1}, there exists a $(d+2)$-angle in $\C$
$$Y_1\rightarrow Y_2\oplus D_1\rightarrow\cdots\rightarrow Y_{d}\oplus D_{d-1}\rightarrow X\oplus D_d \rightarrow \Sigma^d C.$$
Note that $Y_{i+1}\oplus D_i\in\Y$ for $i=1,\cdots,d-1$, $X\oplus D_d\in\X$. Dually, since $\D$ is strongly covariantly finite, for any object $C\in\C$, we can similarly get a $(d+2)$-angle in $\C$
$$C\rightarrow D^\prime_1\oplus Y\rightarrow D^\prime_2\oplus _1X\rightarrow \cdots\rightarrow D^\prime_d\oplus X_{d-1}\rightarrow X_d\rightarrow C$$
with $ D^\prime_i\in\D, X_i\in\X$ for $i=1,\cdots,d$, and $Y\in\Y$. So we prove $(\X,\Y)$ is a cotorsion pair in $\C$. It follows that the map $(\X,\Y)\mapsto(\overline{\X},\overline{\Y})$ is a bijection.
\end{proof}
The following fact is obvious.
\begin{proposition}
Let $\D$ be a strongly functorially finite $d$-rigid subcategory of $\C$  satisfying  ${^\bot}(\Sigma^d\D)=(\Sigma^{-d}\D){^\bot}$, which is denoted by $\Z$. Moreover, we assume $\Z$ is extension closed.  Then the maps $\mu^{-1}_{\D}(-)$ and $\mu_{\D}(-)$ are mutually inverse on the set of subcategories $\M$ of $\C$ satisfying $\D\subset\M\subset\Z$.
\end{proposition}

Now we give the main result of this section.

\begin{theorem}\label{main}
Let $(\X,\Y)$ be a cotorsion pair in $\C$ and $\D\subset\I(\X)$ be a strongly functorially finite subcategory of $\C$ satisfying  ${^\bot}(\Sigma^d\D)=(\Sigma^{-d}\D){^\bot}$, which is denoted by $\Z$. Moreover, we assume that $\Z$ is extension closed. Then the pairs $$(\mu^{-1}_{\D}(\X),\mu^{-1}_{\D}(\Y))\hspace{2mm}\mbox{and}\hspace{3mm}(\mu_{\D}(\X),\mu_{\D}(\Y))$$ are  cotorsion pairs in $\C$ and we have
$$\I(\mu^{-1}_{\D}(\X))=\mu^{-1}_{\D}(\I(\X))\hspace{2mm}\mbox{and}\hspace{3mm}\I(\mu_{\D}(\X))=\mu_{\D}(\I(\X)).$$
\end{theorem}
\begin{proof}
By Remark \ref{l}, we have that $\I(\X)$ is $d$-rigid, then so is $\D$. So both $\mu^{-1}_{\D}(\X)$ and $\mu^{-1}_{\D}(\Y)$ are well-defined. We denote $\mu^{-1}_{\D}(\X)$, $\mu^{-1}_{\D}(\Y)$ and $\mu^{-1}_{\D}(\I(\X))$ by
$\X^\prime$,  $\Y^\prime$ and $\I^\prime$ respectively. As in the previous subsection, we denote by $\mathfrak{U}=\Z/\D$ the subfactor $(d+2)$-angulated category. By Theorem \ref{d}, we have that $(\overline{\X},\overline{\Y})$ is a cotorsion pair in $\mathfrak{U}$ and $\I(\overline{\X})=\overline{\I(\X)}$. Note that its  forward 0-mutation $(T^d\overline{\X},T^d\overline{\Y})$ is also a cotorsion pair in $\mathfrak{U}$. By the construction of $(d+2)$-angles in $\mathfrak{U}$, we have $T^d\overline{\X}=\overline{\X^\prime}$ and $T^d\overline{\Y}=\overline{\Y^\prime}$. Then by Theorem \ref{d}, we have that the pair $(\X^\prime,\Y^\prime)$ is a cotorsion pair in $\C$ and $\I(\overline{\X^\prime})=\overline{\I(\X^\prime)}$.  This means that $\overline{\I^\prime}=T^d\overline{\I(\X)}=T^d\I(\overline{\X})=
\I(\overline{\X^\prime})=\overline{\I(\X^\prime)}$. Therefore, $\I(\mu^{-1}_{\D}(\X))=\mu^{-1}_{\D}(\I(\X))$. The assertion for $\I(\mu_{\D}(\X))=\mu_{\D}(\I(\X))$ can be proved similarly.
\end{proof}



\vspace{0.5cm}

\hspace{-4mm}\textbf{Data Availability}\hspace{2mm} Data sharing not applicable to this article as no datasets were generated or analysed during
the current study.
\vspace{2mm}

\hspace{-4mm}\textbf{Conflict of Interests}\hspace{2mm} The authors declare that they have no conflicts of interest to this work.

\end{document}